\documentclass[x11names,11pt,twoside]{amsart}
\usepackage[utf8]{inputenc}
\usepackage[english]{babel}
\usepackage[T1]{fontenc}
\usepackage{amsfonts}
\usepackage{amsmath, amsthm, amssymb}
\usepackage{latexsym}
\usepackage[bookmarks=true]{hyperref}
\usepackage{cancel}
\usepackage{tikz}

\newcommand*\circled[1]{\tikz[baseline=(C.base)]\node[draw,circle,inner sep=1.2pt,line width=0.2mm,](C) {#1};}

\theoremstyle{plain}
\newtheorem{thm}{Theorem}[section]
\newtheorem{prop}[thm]{Proposition}
\newtheorem{cor}[thm]{Corollary}
\newtheorem{lem}[thm]{Lemma}
\theoremstyle{definition}
\newtheorem{defn}[thm]{Definition}
\newtheorem{ex}[thm]{Example}

\newtheorem{rem}[thm]{Remark}

\newcommand{\Z}{\mathbb{Z}}

\newcommand{\pd}{{\rm pd}}
\newcommand{\lcm}{\mathrm{lcm}}
\newcommand{\gr}{\mathrm{gr}}

\title{The minimal cellular resolutions of the edge ideals of forests}

\author{Margherita Barile, Antonio Macchia}

\address{{\small Margherita Barile, Dipartimento di Matematica, Università degli Studi di Bari Aldo Moro, Via Orabona 4, 70125 Bari, Italy}}
\email{{\small margherita.barile@uniba.it}}

\address{{\small Antonio Macchia, Fachbereich Mathematik und Informatik, Freie Universit\"at Berlin, Arnimallee 2, 14195 Berlin, Germany}}
\email{{\small macchia.antonello@gmail.com, macchia@zedat.fu-berlin.de}}

\begin{document}
\begin{abstract}
We present an explicit construction of a minimal cellular resolution for the edge ideals of forests, based on discrete Morse theory. In particular, the generators of the free modules are subsets of the generators of the modules in the Lyubeznik resolution. This procedure allows to ease the computation of the graded Betti numbers and the projective dimension.
\end{abstract}
\maketitle


\section{Introduction}
Cellular resolutions (free resolutions supported by a CW-complex) were introduced for monomial modules by Bayer and Sturmfels in \cite{BS}. Since then, they have been studied by several authors (among others, \cite{AFG}, \cite{BW}, \cite{DE}, \cite{EN}, \cite{F}, \cite{H}, \cite{M10}, \cite{N}, \cite{V}). In some cases they turn out to be minimal (see, e.g., \cite{BW} for generic and shellable monomial modules, \cite{EN} for the powers of the edge ideals of paths, \cite{M10} for the well-known Eliahou-Kervaire resolution for stable ideals,   \cite{N} for the matroid ideal of a finite projective space). We also know, however, that a minimal cellular free resolution need not exist: a class of counterexamples was provided by Velasco in \cite{V}. A special type of (regular) cellular resolutions are the simplicial resolutions, first considered by Bayer, Peeva and Sturmfels \cite{BPS}. Two famous (in general, non minimal) examples for monomial ideals are the \textit{Taylor resolution} \cite{T} and its refinement called the \textit{Lyubeznik resolution} \cite{L}: an essential overview of the topic is contained in \cite{M11}. In \cite{BW}, Batzies and Welker used Chari's  reformulation \cite{C} of Forman's discrete Morse Theory \cite{Fo} as a tool for obtaining new cellular resolutions from the Taylor complex (they showed that the Lyubeznik resolution can be derived in this way). Later on, \`Alvarez Montaner, Fern\'andez-Ramos and Gimenez \cite{AFG} applied similar techniques for developing an algorithmic transformation (called \textit{pruning}) of the Taylor resolution, which, however, does not always produce a minimal resolution.

In our paper we present an explicit discrete Morse theoretical construction of a minimal cellular free resolution for any edge ideal of a forest, i.e., for any ideal in a polynomial ring over a field whose generators are the squarefree quadratic monomials corresponding to the edges of a (simple, undirected) acyclic graph. The sets of generators of the free modules are explicitly described as special subsets of the sets of generators of the modules in the Lyubeznik resolution. They can be determined in two ways: we present both a selection procedure (see the steps (I)-(V) in Section 3) and a combinatorial characterization (see Proposition \ref{reduced-gaps}). These two methods relevantly facilitate the computation of the (multi)graded Betti numbers and the projective dimension, and, in particular, they allow a transparent constructive approach to some of the formulas given by Jacques \cite{J} and Kimura \cite{K12}. Also note that our method is direct, not recursive, and totally different and independent with respect to the one developed for the quadratic monomial ideals considered by Horwitz \cite{H}.

Since, for edge ideals, the minimal free resolutions are additive with respect to the connected components of the graph, we can restrict our attention to the edge ideals of trees.

\section{Preliminaries}
Let $T$ be a tree on the vertex set $\{x_1,\dots, x_n\}$, which is a set of indeterminates over the field $K$. Let $R=K[x_1,\dots, x_n]$ and let $I=I(T)\subset R$ be the \textit{edge ideal} of $T$, i.e., the ideal generated by the (edge) monomials $x_ix_j$ such that $\{x_i, x_j\}$ (also denoted by $x_ix_j$) is an edge of $T$.

Fix a vertex, say $x_1$. Let $d$ be the maximum distance that a vertex of $T$ can have from $x_1$. For all $i=0,\dots, d$, call $x_1^{(i)},\dots,x_{s_i}^{(i)}$ the vertices lying at distance $i$ from $x_1$ (so that, in particular, $s_0=1$ and $x_1^{(0)}=x_1$). We will say that these vertices have \textit{rank} $i$: each of these vertices is connected to $x_1$ by a (unique) path of length $i$ (formed by $i+1$ vertices).

Neighbours always have consecutive ranks. In the dual graph $\overline{T}$, let $((i)_p, (i+1)_q)$ denote the vertex that, in $\overline{T}$, corresponds to the edge $x_p^{(i)}x_q^{(i+1)}$ of $T$.  The vertices of $\overline{T}$ of the form $((i)_p,-)$ or $(-, (i)_p)$ form a non-empty complete subgraph of $\overline{T}$, whose vertices represent the edges of $T$ that contain the vertex $x_p^{(i)}$. It will be called a $K$-\textit{subgraph} of $\overline{T}$, denoted by $K^{(i)}_p$, and $(i)_p$ will be called its \textit{index}. Similar considerations apply to $(i+1)_q$. The $K$-subgraphs are the maximal non-empty complete subgraphs of $\overline{T}$. Note that, if $x^{(i)}_px^{(i+1)}_q$ is a non-terminal edge monomial, then $((i)_p, (i+1)_q)$ is the only common vertex of $K^{(i)}_p$ and $K^{(i+1)}_q$.

\begin{rem}\label{unique} Let $i\geq 0$. For every index $(i+1)_q$ there is exactly one index of the form $(i)_p$ such that $((i)_p, (i+1)_q)$ is a vertex of $\overline{T}$ (i.e., such that $x^{(i)}_px^{(i+1)}_q$ is an edge of $T$). In fact, if $x^{(i)}_p$ and $x^{(i)}_{p'}$ are different vertices of $T$ of rank $i$, there are paths of length $i$  connecting $x_1$ to each of them. If these vertices were both adjacent to the vertex $x^{(i+1)}_q$ (of rank $i+1$), then $T$ would contain a cycle. We will say that $x^{(i)}_p$ is the only \textit{predecessor} of   $x^{(i+1)}_q$.
\end{rem}

On the monomials of $R$ fix the lexicographic order induced by the following arrangement of the indeterminates:
\begin{equation*}
x_1^{(0)}>x_1^{(1)}>\cdots>x_{s_1}^{(1)}>x_1^{(2)}>\cdots>x_{s_2}^{(2)}>\cdots> x_{s_d}^{(d)}.
\end{equation*}
\noindent
The arrangement thus obtained will be called the \textit{sequence of generators}.\\
The same order will be considered on the set of indices of the $K$-subgraphs.

\section{The resolution}

\subsection{The generators of the modules}

Let $S$ be the sequence of generators of $I$.
Any subsequence of $S$ will be called a \textit{symbol}, and will be written as a list of (pairwise distinct) elements, separated by commas, in round brackets.

Given a symbol $u=(\mu_1, \dots, \mu_r)$, $r$ will be called the \textit{length} of $u$, and denoted by $\vert u\vert$. Every subsequence of $u$ will be called a \textit{subsymbol} of $u$ (or a symbol \textit{contained} in $u$). We will also say that $\mu_1, \dots, \mu_r$ are the \textit{elements} of (belong to) $u$. We will thus treat $u$ as an ordered set.  We will set $\lcm(u)=\lcm(\mu_1,\dots, \mu_r)$.

We recall the following definition, which Lyubeznik \cite{L} gave for an arbitrary monomial ideal.

\begin{defn}\label{definition5}
A symbol $u=(\mu_1,\dots, \mu_r)$ is called $L$-\textit{admissible}  if $\mu_q$ does not divide $\lcm(\mu_{i_h},\mu_{i_{h+1}},\dots,\mu_{i_t})$ for any $h<t$ such that $q<i_h$.
\newline
It is called \textit{reduced} if $\mu_q$ does not divide $\lcm(\mu_1,\dots,\widehat{\mu_q},\dots,\mu_r)$ for all $q$ such that $1\leq q\leq r$.
\end{defn}

\begin{rem}
In the special case that we are considering here, i.e., when $I$ is generated by squarefree monomials of degree two, the condition of $L$-admissibility can be restated in the following simpler form: $\mu_q$ does not divide any product $\mu_{i_h}\mu_{i_{k}}$ for any $h, k$ such that $h<k\leq t$ and $q<i_h$. Similarly, the condition of being reduced translates into: $\mu_q$ does not divide any product $\mu_{i_h}\mu_{i_{k}}$ for any $h, k$ such that $q\ne i_h$, $q\ne i_k$. Thus being reduced implies $L$-admissibility.
\end{rem}

Set $L_0=R$ and, for all $r=1,\dots,\vert S\vert$, let $L_r$ be the free $R$-module generated by all $L$-admissible symbols of length $r$. Define the map $\delta_r:L_r\to L_{r-1}$ by setting
\begin{equation*}\label{d}
\delta_r((\mu_{i_1},\dots,\mu_{i_r}))\!=\!\! \sum_{j=1}^r \! (-1)^{j\!+\!1} \! \frac{\lcm(\mu_{i_1},\dots,\mu_{i_r})}{\lcm(\mu_{i_1},\dots\!,\widehat{\mu_{i_j}},\dots\!,\mu_{i_r})} (\mu_{i_1},\dots,\widehat{\mu_{i_j}},\dots,\mu_{i_r}).
\end{equation*}
Then one has the following

\begin{thm}[\cite{L}, p.~193]\label{Lresolution}
The complex
\begin{equation}
\tag{$\star$} 0 \longrightarrow L_s \stackrel{\delta_s}{\longrightarrow} L_{s-1} \xrightarrow{\delta_{s-1}} \cdots \stackrel{\delta_1}{\longrightarrow} L_0 \longrightarrow 0
\end{equation}
is a free resolution of $R/I$.
\end{thm}

\noindent
The resolution $(\star)$ is called a \textit{Lyubeznik resolution} of $I$.

\medskip
The Lyubeznik resolution is a refinement of the \textit{Taylor resolution}, whose $r$th module is generated by all symbols of length $r$, and in general it is not minimal. For all indices $r$ we will determine a submodule $F_r$ of $L_r$ such that the resulting resolution is always minimal for the edge ideals of trees. The differentials, however, will have to be redefined by means of discrete Morse theory. The submodules $F_r$ are generated by special $L$-admissible symbols, which we will call $F$-admissible. We now start the description of a procedure for selecting these symbols. The first steps are presented here below.
\begin{itemize}
\item[(I)] Select a descending sequence of indices $(i_1)_{p_1},\dots,(i_t)_{p_t}$ that does not contain any pair of indices corresponding to adjacent vertices in $T$.
\item[(II)] Pick all edge monomials corresponding to the vertices of the $K$-subgraphs $K^{(i_1)}_{p_1},\dots,K^{(i_t)}_{p_t}$.
\item[(III)] In the symbol thus obtained, cancel all monomials $\mu$ fulfilling the following condition: $\mu$  corresponds to a vertex of  $K^{(i_h)}_{p_h}$ and is not coprime with respect to an element $\nu$ of the symbol corresponding to a vertex of  $K^{(i_k)}_{p_k}$ for some $k>h$.
\item[(IV)] Consider all subsymbols of the symbols thus obtained.
 \end{itemize}

\begin{defn}
A symbol arising from the procedure (I)--(IV) will  be called \textit{almost F-admissible}. The set of the monomials of an almost $F$-admissible symbol that are divisible by $x^{(i_j)}_{p_j}$ - i.e., those corresponding to the vertices of $K^{(i_j)}_{p_j}$ - will be called the $(i_j)_{p_j}$-\textit{block} of $u$. The number $(i_j)_{p_j}$ will be called its \textit{index}.
\end{defn}

\begin{rem}
Note that the sequence of the indices of the blocks of an almost $F$-admissible  symbol $u$ is not always uniquely determined. This is particularly evident in the case where $u$ is formed by one single monomial $x^{(i)}_px^{(i+1)}_q$ not corresponding to a terminal edge of $T$; then $u$ can be, of course, indifferently assigned to the $(i)_p$-block or to the $(i+1)_q$-block of $u$.
\end{rem}

\begin{rem}\label{disjoint}
As a consequence of the selection performed at step (III), any two monomials of an almost $F$-admissible symbol $u$ belonging to different blocks are coprime. In particular, if two distinct monomials of $u$ are divisible by $x^{(i)}_p$, then they belong to the $(i)_p$-block of $u$.
\end{rem}

\begin{prop}\label{reduced-admissible}
Every almost $F$-admissible symbol is reduced. In particular it is $L$-admissible.
\end{prop}

\begin{proof}
Let $u$ be an almost $F$-admissible symbol. In view of Remark \ref{disjoint}, if the edge monomial $\mu=x^{(i)}_px^{(i+1)}_q$  belongs to the $(i)_p$-block of $u$, no other monomial of $u$ is divisible by $x^{(i+1)}_q$. If it belongs to the $(i+1)_q$-block of $u$, no other monomial of $u$ is divisible by $x^{(i)}_p$. Thus $\mu$ cannot divide the product of any other two monomials of $u$.
\end{proof}

\begin{defn}\label{gap-defn}
Let $x,y,z,w$ be four distinct vertices such that $xy$ and $zw$ are elements of the symbol $u$, where $xy>zw$. If $xz\in I$, we say that $xz$ is the \textit{bridge} between $xy$ and $zw$. In this case, we will say that $xy$ and $zw$ form a \textit{gap} in $u$ if $xz\notin u$, no other monomial of $u$ other than $zw$ is divisible by $w$, and no monomial smaller than $zw$ is divisible by $y$.
The monomial $xz$ will be called the \textit{bridge} of the gap.
\end{defn}

We will say that a symbol contains a bridge if it contains a triple of monomials $xy, zw, xz$, which occurs if and only if the symbol is not reduced.

\begin{rem}\label{ongaps}
(a) In the assumption of Definition \ref{gap-defn}, the condition $xz\notin u$ is always fulfilled if $u$ is reduced.\newline
(b) In  Definition \ref{gap-defn} we are not assuming that $x>y$. In this case the inequality $xy>zw$ implies that $x>z$. The same is true, however, if $y>x$, since $y$ is then the only predecessor if $x$. Hence $x$ is always the predecessor of $z$, which, in turn, implies that $z>w$. Finally, we deduce that $xz>zw$.
\end{rem}

We now present the last step of the procedure.
\begin{itemize}
\item[(V)] Discard all symbols that contain a gap.
\end{itemize}

\begin{defn}\label{adm-defn}
The symbols obtained after completion of the procedure (I)--(V) will be called $F$-\textit{admissible} (for $T$).
\end{defn}

\begin{rem} One may ask whether the steps (III) and (IV) could be interchanged. This would enlarge the set of almost $F$-admissible symbols: if the indices $(i_h)_{p_h}$ and $(i_k)_{p_k}$  with $k>h$ have been selected at step (I), and  $\mu=xy$ (where $x=x^{(i_h)}_{p_h}$)  and $\nu=xz$ (where $z=x^{(i_k)}_{p_k}$) are monomials corresponding to vertices of $K^{(i_h)}_{p_h}$ and  $K^{(i_k)}_{p_k}$, then, according to step (III), $\mu$ should be deleted. This would not occur, however, if, after selecting a subsymbol (as in step (IV)), $\nu$ no longer appeared. Suppose that the $(i_k)_{p_k}$-block  is still represented in the subsymbol by some monomial $zw$ (otherwise we could regard the index $(i_k)_{p_k}$ as not selected in step (I)). We would thus have a symbol that contains $\mu$ but not $\nu$, a possibility that is forbidden by the original procedure.  But the point is that $xy$ and $zw$ form a gap, so that the subsymbol would be discarded in step (V).  In other words, performing steps (III) and (IV) in this order makes the selection procedure more efficient.
\end{rem}

The $F$-admissible symbols can be characterized combinatorially in the following way.

\begin{prop}\label{reduced-gaps}
A symbol is $F$-admissible if and only if it is reduced and has no gaps.
\end{prop}

\begin{proof} The ``only if'' part is a consequence of Definition \ref{adm-defn} and Proposition \ref{reduced-admissible}. For the ``if'' part, let $u$ be a symbol that is reduced and has no gaps. We show that its blocks fulfil the requirements set by steps (I) and (III) of the above procedure. First suppose that the elements of $u$ are pairwise coprime. Then the requirement in (III) is fulfilled. Let $x,y,z,w$ be vertices of $T$ such that $xy$ and $zw$ are two distinct elements of $u$, where $xy>zw$. Then none of $x, y$ is a neighbour of $z$ or $w$, which immediately yields the requirement in (I): if this were not true, and, say, $xz$ were an edge monomial, then $xy$ and $zw$ would form a gap (against our assumption) because no other monomial of $u$ is divisible by $y$ or $w$.\newline
Now assume that, for some indeterminate $z=x^{(i)}_p$, $u$ contains more than one monomial divisible by $z$, say $zw_1$ and $zw_2$ are two distinct elements of $u$. Then no other monomial of $u$ is divisible by $w_1$ or $w_2$, because otherwise $u$ would not be reduced.  We show that, up to changes in the way in which the monomials of $u$ are assigned to its blocks, no monomial of $u$ belonging to a block preceding the $(i)_p$-block can be divisible by $z$. Suppose by contradiction that some element of the $(i')_{p'}$-block of $u$, where $(i')_{p'}>(i)_p$, is divisible by $z$. Let $x=x^{(i')}_{p'}$. Then $xz\in u$, and no other monomial of $u$ can be divisible by $x$, because otherwise $u$ would not be reduced. Hence the monomial $xz$ can be assigned to the $(i)_p$-block of $u$, and the $(i')_{p'}$-block disappears. If this transformation is applied repeatedly, scanning the blocks of $u$ from right to left, then after a finite number of steps the condition required at step (III) is achieved and, in particular, the monomials belonging to different blocks are coprime. \newline
Now assume that the sequence of indices of the blocks of $u$ contains two indices $(i')_{p'}>(i)_p$ corresponding to consecutive vertices $x$ and $z$. Let $xy$ and $zw$ be elements of the $(i')_{p'}$-block and the $(i)_{p}$-block, respectively. Then no other monomial of $u$ can be divisible by $y$ or $w$, so that $xy$ and $zw$ form a gap. But this contradicts our assumption. Hence $u$ also fulfils the requirement in (I).
\end{proof}

\subsection{The complex}

For all  indices $r$, let $F_r$ be the free $R$-module generated by the $F$-admissible symbols of length $r$. We show that $F_r$ is the $r$th module of a minimal graded free resolution of $R/I$: this is the cellular resolution derived from the Taylor resolution by means of the construction described in Section 1 of \cite{BW}. Note that the Taylor resolution can be viewed as a simplicial complex, and therefore as a $\Z^n$-graded regular CW-complex $(X, {\rm gr})$, where, for all indices $r$, the $r$-cells are the symbols of length ({\it dimension}) $r$, and the $\Z^n$-grading is defined as follows: ${\rm gr}(\mu_1,\dots, \mu_r)=\sum_{k=1}^s{e_{j_k}}$ if $\lcm(\mu_1,\dots, \mu_r)=\prod_{k=1}^sx_{j_k}$ (here $e_j$ denotes the $j$th element of the canonical basis of $\Z^n$). We also endow the Cartesian product $\Z^n$ with the usual termwise defined partial order, with respect to which, for any two symbols $u$ and $v$, we have gr\,$(u)\leq\,$gr\,$(v)$ if and only if $\lcm(u)$ divides $\lcm(v)$.

We then consider the directed graph $G_X$ on $X$ whose set of edges $E_X$ is formed by the directed edges $u\to u'$ such that $u'\subset u$ and the lengths of $u$ and $u'$ differ by one. We consider the set of symbols that are not $F$-admissible, i.e., according to Proposition \ref{reduced-gaps}, that either contain a gap or are not reduced.

With respect to the notation of Definition \ref{gap-defn}, given a symbol $u$ containing the gap formed by $xy>zw$ (whose bridge is $xz$), we say that a bridge $\lambda$ of $u$ {\it follows} this gap if $xz>\lambda$.

We will call of {\it type} 1 every symbol containing a gap that is not followed by any bridge, and of {\it type} 2 any other non-$F$-admissible symbol. Note that any symbol of type 2  has no gaps (and then it is not reduced) or it contains a bridge that follows all its gaps. In any case it contains a bridge.

Let $A$ be the set of directed edges $u\to u'$ such that $u'$ is of type 1   and $u$ is obtained from $u'$ by inserting the monomial (say $xz$) which is the bridge of one of its gaps (say $xy>zw$) and is the smallest among the bridges of its gaps. We will call this operation {\it bridge insertion}.

\begin{lem}\label{propType1Type2Symbols}
With respect to the above notation, the following properties hold:
\begin{itemize}
\item[1)] The symbol  $u$ is of type $2$ (thus $u$ can never appear as the first vertex of an edge of $A$).
\item[2)] The symbol $u'$ is obtained from $u$ by omitting its smallest bridge (thus $u$ appears in exactly one edge of $A$).
\item[3)] If a symbol is of type $2$, it can be obtained by bridge insertion from a symbol of type $1$ (thus any symbol of type $2$ appears in some edge of $A$).
\end{itemize}
\end{lem}

\begin{proof}
1) Suppose by contradiction that $u$ contains some gap $ab>cd$ (with bridge $bc$) which is not followed by any bridge.  Thus $xz>bc$. Moreover, by definition of gap, no monomial of the form $de$ (where $d>e$) can belong to $u$. If both $ab$ and $cd$ belonged to $u'$, then they would form a gap also in $u'$, which would contradict the assumption on $u'$. Hence one of $ab$ and $cd$ is $xz$. But $cd\ne xz$, because, by Remark \ref{ongaps} (b), $bc>cd$. Thus $cd\in u'$ and we must have that $ab=xz$. Then $b=x$ or $b=z$. Recall that by Remark \ref{ongaps} (b) we have $x>z>w$ and $b>c>d$. If $bc=xc$, then from $xz>xc$ we deduce that $z>c$, so that $zw>cd$ and, by definition of gap (the one formed by $xy>zw$ in $u'$), no monomial containing $y$ can follow $cd$.  This implies that $xy>cd$ is a gap in $u'$. But its bridge $xc$ is smaller than $xz$, a contradiction. Now suppose that $bc=zc$, so that $xz>zc$. Note that $x>z$ implies $z>c$. Thus $zw>cd$ is a gap in $u'$ (recall that no monomial of $u'$ other than $zw$ is divisible by $w$), and its bridge $zc$ is smaller than $xz$, a contradiction.

\smallskip
2) By assumption, $xz$ is smaller than all bridges of $u'$. We have to show that it is smaller than all monomials of $u'$ that are bridges in $u$, but are not such in $u'$. The point is that the insertion of $xz$ can produce new bridges, since a monomial $\mu$ of $u'$ of the form $ax$ or $az$ could become a bridge between $xz$ and some other monomial $ab$ of $u'$. We show that in this case $\mu>xz$. First suppose that $\mu=az$. Note that $a\neq w$, because, by definition of gap, $zw$ is the only monomial of $u'$ containing $w$, and $ab\in u'$. But then $az$ is already a bridge in $u'$ (between $zw$ and $ab$). Now suppose that $\mu=ax$. If $a\neq y$, then $ax$ is a bridge in $u'$ between $xy$ and $ab$. So assume that $a=y$, i.e., $\mu=xy$. If $x>y$, then $y>b$.  On the other hand, by definition of gap (applied to $xy>zw$), we have $yb>zw$, which implies that $y>z$, so that $\mu>xz$. If $y>x$, then $x>z$ implies, once again, that $\mu>xz$.

\smallskip
3) We show that, if in some symbol $u$ of type 2 we cancel the smallest bridge, say $xz$, which is the bridge between $xy$ and $zw$ (where $xy>zw$), then, in the resulting monomial $u'$, these two monomials form a gap (which, of course, is not followed by any bridge). We have to verify the following two conditions:
\begin{itemize}
\item[i)] $u$ does not contain any other monomial of the  form $wb$. Suppose by contradiction that $u$ contains such a monomial. Then $zw$ is the bridge between $xz$ and $wb$. But from Remark \ref{ongaps} (b) we know that $xz>zw$, against the minimality of $xz$.
\item[ii)] $u$ does not contain any monomial of the form $ya$ that is smaller than $zw$. Suppose that $u$ contains some monomial $ya$. We show that $ya>zw$. This is clear if $y>x$, because $x>z$. So assume that $x>y$, whence $y>a$. Since $xy$ is the bridge between $xz$ and $ya$, we have that $xy>xz$.  Thus $y>z$, which immediately implies that $ya>zw$, as desired.\qedhere
\end{itemize}
\end{proof}

\begin{prop}\label{acyclic}
The graph $G^A_X$ with edge set
\[
E^A_X=(E_X\setminus A)\cup\{u'\to u \, \vert\, u\to u'\in A\}
\]
does not contain any directed cycle.
\end{prop}

\begin{proof}
 The (directed) edges of $E^A_X$  are of the following two types:
\begin{itemize}
\item[(a)] the edges $u_1\to u_2$ not belonging to $A$, where $u_2$ is obtained from $u_1$ by deleting a monomial;
\item[(b)] the edges $u'\to u$ where $u'$ is of type 1, $u$ is of type 2 and $u$ is obtained from $u'$ by bridge insertion.
\end{itemize}

Note that, with respect to the above notation, $\vert u_2\vert=\vert u_1\vert-1$, whereas $\vert u\vert=\vert u'\vert+1$. It follows that any directed cycle of $G^A_X$ must contain at least one edge of each type. More precisely, since every edge of type (b) (along which the length grows by 1) is followed by an edge of type (a) (along which the length drops by 1), every directed cycle consists of an alternating sequence of edges of types (a) and (b), and an alternating sequence of vertices of types 1 and 2.

Moreover, $\gr(u')=\gr(u)$, whereas $\gr(u_1) \geq \gr(u_2)$ and equality holds if and only if $u_2$ is obtained from $u_1$ by deleting a bridge. Thus, in any directed cycle of $G^A_X$ all vertices have the same degree and two consecutive vertices always differ by a bridge. Let $C$ be a directed cycle of $G^A_X$. Let $u'\to u$ be a directed edge of $C$ of type (b), where $u$ is obtained from $u'$ by inserting the bridge $xz$ of the gap $xy>zw$. Moreover, assume that, among all edges of type (b) of $C$, this edge is one for which $xz$ is maximum. The cycle $C$ also contains an edge $v\to u'$ of type (a), where $v$ is obtained from $u'$ by inserting  a bridge other than $xz$ (because $v\neq u$). In particular, $xz$ does not belong to $v$. This implies that at some point of the directed path  of $C$ from $u$ to $v$ the bridge $xz$ is deleted. The first edge of this path, say $u\to u_1$, is of type (a), and $u_1\neq u'$, so that $u_1$ is of type 1 and $xz\in u_1$. Hence in $u_1$ there are two monomials $ab>cd$ forming a gap with bridge $bc$ (the smallest), where, by choice of $u'$, $bc<xz$.  Since by Remark \ref{ongaps} (b) we have $bc>cd$, we also have that $cd\ne xz$, whence $cd\in u'$, because $cd\in u_1\subset u$.

Moreover, as we have seen in the proof of Lemma \ref{propType1Type2Symbols} 1), the condition $ab=xz$ is incompatible with the fact that $xz$ is the smallest of the bridges of the gaps contained in $u'$. Hence $ab$ and $cd$ both belong to $u'$, where, of course, they cannot form a gap, because $xz>bc$. Since $u'\subset u$ (so that $ab, cd\in u$), for the same reason we have that $bc\notin u$ (otherwise $bc$ would be a bridge of $u$ smaller than $xz$), whence $bc\notin u'$. Thus the obstruction preventing $ab>cd$ from forming a gap in $u'$ must be due to the presence of some other monomials forbidden by the definition of gap, i.e., in $u'$ there is some other monomial $\mu=de$ (with $d>e$) or some monomial $\mu=af$ that is smaller than $cd$.

Since $\mu$ does not belong to $u_1$, it must have been deleted along the path from $u$ to $u_1$, hence $\mu$ is a bridge in $u$, so that $\mu>xz>bc$. These inequalities exclude the case $\mu=de$.  If $\mu=af$, since $cd>af$, we have $c>a$, whence $b>a$, and $a>f$. Thus $bc>\mu$, a contradiction. This shows that $G^A_X$ cannot contain any directed cycle.
\end{proof}

As a consequence of Lemma \ref{propType1Type2Symbols} and Proposition \ref{acyclic}, $A$ is a so-called \textit{acyclic matching} for $G_X$. According to \cite[Proposition 1.2]{BW}, this implies that there is a $\Z^n$-graded CW-complex (called the \textit{Morse complex} of the Taylor resolution), which has the following two properties: its $r$-cells are in one-to-one correspondence with the $r$-cells of $X$ not belonging to any edge of $A$ (which are called \textit{critical cells} in \cite{BW}), i.e., with the $F$-admissible symbols of length $r$,  and it is homotopy equivalent to the Taylor resolution. This Morse complex supports a cellular resolution of $R/I$. An explicit description of its differentials will be provided later on. We first show its minimality.
In view of \cite[Corollary 7.6]{BW}, our claim will follow from the next result.

\begin{lem}\label{divisibility}
Let $u$ and $v$ be distinct $F$-admissible symbols. Then $\gr(u) \neq \gr(v)$.
\end{lem}

\begin{proof}  First note  that whenever some reduced symbol is contained in some other reduced symbol, the grade of the former is strictly smaller than  the grade of the latter. So assume that $u$ and $v$ are incomparable by inclusion. Note that it suffices to prove the claim under the assumption that $\gr(u) \geq \gr(v)$, i.e., that $\lcm(v)$ divides $\lcm(u)$. For every monomial $ab$ of $u$ (or $v$) such that $a>b$, we call $b$ a successor of $u$ (of $v$). Since every vertex of the tree $T$ has at most one  predecessor, whenever $b$ is a successor of $u$, $ab$ belongs to $u$. Therefore, since $u$ is not contained in $v$, there is some successor of $u$  that is not a successor of $v$. Let $b$ (whose predecessor is $a$) be the smallest such successor of $u$. Then $ab\in u$, but $ab\notin v$. We show that $ab$ does not divide $\lcm(v)$. This will immediately yield the claim.\newline
Suppose by contradiction that $ab$ divides $\lcm(v)$. Since $b$ divides $\lcm(v)$, it follows that $bc\in v$ for some $c<b$. Thus $c$ divides $\lcm(u)$. If $bc\notin u$, then $cd\in u$ for some $d<c$. Since $ab>cd$ and these two monomials do not form a gap in $u$, we must have $de\in u$ for some $e<d$. Note that $d$ and $e$ are successors of $u$ smaller than $b$, whence $cd, de\in v$. But this, together with $bc\in v$, would imply that $v$ is not reduced. We thus conclude that $bc\in u$.\newline
On the other hand, $a$ divides $\lcm(v)$, which implies that $xa\in v$ for some $x\neq b$. The vertex $x$ may be greater or smaller than $a$, but in any case $xa>bc$. Since these two monomials do not form a gap in $v$, we have one of the following cases:
\begin{itemize}
\item[i)] $cd\in v$, for some $d<c$. Then, as above, we conclude that $cd\in u$: it suffices to apply the argumentation developed for the monomials $ab\in u$ and $bc\in v$ to the monomials $bc\in u$, and $cd\in v$. But then all three monomials $ab, bc$ and $cd$ belong to $u$, which is impossible.
\item[ii)] $b>x$ (in which case $a>x$) and $xy\in v$ for some $y<x$. Then $u$ contains some monomial divisible by $x$, either $ax$ or $xz$, with $z<x$. If $ax\in u$, we conclude as above that $xy\in u$, which is incompatible with the fact that $ab\in u$. So suppose that $xz\in u$. Since $ab>xz$ and these monomials do not form a gap in $u$, we have that $zw\in u$ for some $w<z$. Since both $z$ and $w$ are successors of $u$ smaller than $b$, it follows that $xz$ and $zw$ belong to $v$, which, together with $ax\in v$, would once again cause $v$ to be not reduced.
\end{itemize}
In any case we have a contradiction.
\end{proof}
\begin{rem} The preceding lemma yields the so-called \textit{dual version of Hochster's formula} (see, e.g., \cite[Corollary 1.40]{MS}) on squarefree monomial ideals in the special case of edge ideals of trees: it proves that the nonzero multigraded Betti numbers all lie in squarefree degrees. It also gives the result in \cite[Theorem 3.5]{EF}, which Erey and Faridi proved for the more general class of simplicial forests: for every multidegree there is at most one nonzero Betti number, and this is equal to 1.
\end{rem}
The following notation is taken from \cite{BW}. For every pair $(u,u')$ of $F$-admissible symbols with $r=\vert u\vert=\vert u'\vert+1$, call $[u:u']$ the coefficient of $u'$ in $\delta_r(u)$. If the directed edge $u\to u'$ belongs to $A$, we then set $m(\{u,u'\})=-[u:u']$, otherwise we set $m(\{u,u'\})=[u:u']$. Given a directed path $P:u_0\to u_1\to\cdots\to u_t$ in $G^A_X$ (a so-called \textit{gradient path} from $u_0$ to $u_t$), we also set $m(P)=\prod_{i=0}^{t-1}m(\{u_i, u_{i+1}\})$.  Note that $\gr(u_0)\geq\gr(u_t)$. We can now define the $r$th differential $\partial_r: F_r\rightarrow F_{r-1}$ of our resolution. According to \cite[Lemma 7.7]{BW}, for every $F$-admissible symbol $u$ of length $r$,
\begin{equation}\label{diff}
\partial_r(u) \!=\!\! \sum_{\substack{u'\subset u \\ \vert u'\vert =r-1}} [u:u'] \sum_{\substack{u'' F-\text{admissible}, \\ \vert u''\vert=r-1}} \ \sum_{\substack{P \text{ gradient path}\\ \text{from } u' \text{ to } u''}}m(P)\underline{x}^{\mbox{\small gr}\,(u)-\mbox{\small gr}\,(u'')}u''.
\end{equation}

We have just established the following
\begin{thm}\label{main}
The cellular resolution $(F_r, \partial_r)$ is a minimal graded free resolution of $R/I$.
\end{thm}

\begin{proof}
We just have to observe that minimality is ensured by Lemma \ref{divisibility}, since, according to \cite[Corollary 7.6]{BW} and \cite[Proposition 7.3]{BW}, a sufficient condition is the following: one has that gr\,$(u)\neq\,$gr\,$(v)$ for all $F$-admissible symbols $u$ and $v$ such that $\vert u\vert=\vert v\vert +1$ and either $v\subset u$ or there exists a gradient path from a symbol $u'\subset u$ of length $\vert v\vert$ to $v$.
\end{proof}
\begin{rem} Note that, according to (\ref{diff}), the morphism $\partial_r$ sends every $F$-admissible symbol $u$ of length $r$ to a linear combination (with monomial coefficients) of $F$-admissible symbols $u''$ of length $r-1$ such that $\lcm(u'')$ divides $\lcm(u)$. Among these we find all $u''$ that are contained in $u$ (those appearing in the definition of the morphism $\delta_r$ of the Lyubeznik resolution, and for which the gradient path $P$ consists of one single directed edge $u\to u''$), but possibly some others. This will become evident in Example \ref{u''}.
\end{rem}

\section{Examples and further remarks}
We first give a concrete example of determination of the modules of the cellular minimal free resolution of the edge ideal of a tree. We will apply the selection procedure described in Section 3. In order to simplify our notation, we will replace the symbol $x^{(i)}_p$ by the number $i$ with $p$ apices. In this way we will also write the index of the $K$-subgraph $K^{(i)}_p$ and of the $(i)_p$-block.
\begin{ex} Let $T$ be the tree on the vertices $0>1> 1'>2>2'>2''>3$ whose edges are $01$, $01'$, $12, 12'$, $1'2''$, $23$. Its $K$-subgraphs are:
$$ [0]: 01, 01'\qquad [1]: 01, 12, 12'\qquad [1']: 01', 1'2''\qquad [2]: 12, 23.$$
In the next table we list, for any choice of the sequences of indices in step (I) (written in square brackets), the edges of the corresponding $K$-subgraphs, as prescribed in step (II), and perform on them the cancellations indicated in (III). Then we consider, as in step (IV), all their subsymbols. Each horizontal section refers to the sequences of a given length (one in each column). In each column, the subsymbols will be arranged in descending order of length (denoted by $r$).  We will avoid repetitions by omitting the subsymbols already obtained in the preceding sections, and replacing the others by a reference to the number of the column of the same section (expressed by ($\ast$) or $(\ast\ast)$) where (in the same row) they appear for the first time. The monomials forming a gap are boxed, and the vertices of the bridge are in bold type. According to step (V), the corresponding subsymbols have to be discarded.
\vskip.3truecm
\begin{tabular}{|l|l|l|l|l|}
\hline
& $[0]$&$[1]$&$[1']$ &$[2]$\\
\hline
$r=3$&&$(01, 12, 12')$&&\\
\hline
$r=2$&$(01, 01')$&$(01, 12)$&$(01', 1'2'')$&$(12, 23)$\\
&&$(01, 12')$&&\\
&&$(12, 12')$&&\\
\hline
$r=1$&$(01)$&$(\ast)$&&\\
& $(01')$& $(12)$& $(\ast)$&$(\ast\ast)$\\
&& $(12')$& $(1'2'')$&$(23)$\\
\hline
\end{tabular}\\

\begin{tabular}{|l|l|l|l|}
\hline
&$[0,2]$ &$[1,1']$&$[1',2]$\\
\hline
$r=4$&&$(\cancel{01},01', 12, 12', 1'2'')$&$(01', 12, 1'2'', 23)$\\
\hline
$r=3$&$(\cancel{01}, 01', 12, 23)$&$(\boxed{{\bf 0}1'},\boxed{{\bf1}2}, 12')$&$(\ast)$\\
&&$(01', 12, 1'2'')$&$(\ast\ast)$\\
&&$(01', 12', 1'2'')$&$(01', 1'2'',23)$\\
&&$(12, 12', 1'2'')$&$(12, 1'2'',23)$\\
\hline
$r=2$&$(\boxed{{\bf0}1'}, \boxed{{\bf1}2})$&$(\ast)$&$(\ast)$\\
&$(01',23)$& $(\boxed{{\bf0}1'}, \boxed{{\bf1}2'})$& $(\ast)$\\
&& $(12,1'2'')$& $(\ast\ast)$\\
&&$(12',1'2'')$& $(1'2'', 23)$\\
\hline
\end{tabular}
\vskip.3truecm
If we count lengths $r$ and degrees $d$ in order to compute the graded Betti numbers $\beta_{r,d}(R/I)$, we exactly obtain the numerical resolution provided by CoCoA \cite{ABL}:
$$0\longrightarrow R(-6)^2\longrightarrow R(-4)\oplus R(-5)^6\longrightarrow R(-3)^6\oplus R(-4)^4\longrightarrow R(-2)^6\longrightarrow 0.$$
Note that the vertex 0 has been chosen at random. The result is independent of this choice.
\end{ex}
We next present two gradient paths from a given $F$-admissible symbol $u$ of length $r$ to two $F$-admissible symbols $u''_1$ and $u''_2$ of length $r-1$.
\begin{ex}\label{u''}
Let $T$ be the tree on the vertices $0>1>2>3>4>4'>5>6$, whose edges are $01$, $12$, $23$, $34$, $34'$, $45$, $56$. Then $u=(01, 23, 34', 45, 56)$ is an $F$-admissible symbol of length 5. In the following table, the cancellations correspond to directed edges of type (a), the insertions to directed edges of type (b), according to the classification introduced at the beginning of the proof of Proposition \ref{acyclic}. The first step is a cancellation that takes $u$ to a symbol of type 1 (which is reduced and contains a gap). The following symbols are alternatively of types 2 and 1, of lengths $r$ and $r-1$, and the procedure stops eventually, when an edge of type (a) leads to an $F$-admissible symbol of length $r-1$. The vertices in bold type are those forming the bridge of the smallest gap in a symbol of type 1 (and thus forming the monomial that will be inserted in the subsequent step).
\vskip.3truecm
\begin{tabular}{|l|l|l|l|l|l|l|l|}
\hline
$u$&$01$&&$23$&$34'$&&$45$&$56$\\
&&&&&&&$\, \, \downarrow$\\
\hline
type 1&$01$&&$23$&${\bf3}4'$&&${\bf4}5$&$ \, \square$\\
&&&&&$\ \, \downarrow$&&\\
\hline
type 2&$01$&&$23$&$34'$&$\circled{\bf 34}$&$45$&\\
&&&&&&$\, \, \downarrow$&\\
\hline
$u''_1$&$01$&&$23$&$34'$&$34$&$\, \square$&\\
\hline
\end{tabular}

\begin{tabular}{|l|l|l|l|l|l|l|l|}
\hline
$u$&$01$&&$23$&$34'$&&$45$&$56$\\
&&&&$\, \, \downarrow$&&&\\
\hline
type 1&$0{\bf1}$&&$ {\bf2}3$&$\, \square$&&$45$&$56$\\
&&$\, \ \downarrow$&&&&&\\
\hline
type 2&$01$&$\circled{\bf 12}$&$23$&&&$45$&$56$\\
&&&&&&&$\, \, \downarrow$\\
\hline
type 1&$01$&$12$&$2{\bf3}$&&&${\bf4}5$&$\, \square$\\
&&&&&$\ \, \downarrow$&&\\
\hline
type 2&$01$&$12$&$23$&&$\circled{\bf 34}$&$45$&\\
&&&$\ \, \downarrow$&&&&\\
\hline
$u''_2$&$01$&$12$&$\, \square$&&$34$&$45$&\\
\hline
\end{tabular}
\end{ex}
\vskip.3truecm
From the description of the minimal cellular resolution of $R/I(T)$ we can derive a formula to compute its projective dimension. Recall that a \textit{leaf} or \textit{free vertex} of a graph is a vertex of degree one.

\begin{lem}\label{max-u}
Let $T$ be a tree, and let $u$ be a maximal $F$-admissible symbol for $T$. Then
\begin{equation*}
\vert u\vert=\#\!\left\{\text{\rm leaves of $T$ in the blocks of $u$}\right\} + \#\!\{\text{\rm blocks not in $u$}\}.
\end{equation*}
\end{lem}

\begin{proof}
Let $u$ be a maximal $F$-admissible symbol. We show that $u$ has length equal to the right-hand side of the formula in the statement.

Let $(a)$ be a block of $u$ that has been selected at step (I) and let $ab$ be a free vertex of the block $(a)$ in the dual graph of $T$, i.e. $b$ is a leaf of $T$. Then the monomial $ab$ is not canceled at step (III) since no monomial divisible by $a$ or $b$ appears in any of the following blocks $(c)$, with $a>c$.

Now, let $(a)$ be a block not in $u$, i.e. that has not been selected at step (I). Let $\mu_1,\dots,\mu_k$ be the vertices of the block $(a)$ that belong also to some other block that has been selected at step (I), say, $\mu_i$ belongs to the block $(b_i)$, i.e. $\mu_i=ab_i$ for $i=1,\dots,k$, with $b_1>b_2>\cdots>b_k$. (Notice that if a vertex is not free and belongs to some block of $u$, then it also belongs to some non-selected block.) Then $ab_1,\dots,ab_{k-1}$ are removed at step (III), while $ab_k$ is not canceled (notice that $ab_k$ is not a free vertex of the dual graph, since it belongs to two blocks).

The case in which all vertices of the block $(a)$ do not belong to any block of $u$ cannot happen, otherwise the symbol $u$ would not be maximal since adding the block $(a)$ to $u$ would produce a larger symbol.
\end{proof}

\begin{rem}
Notice that the characterization in Lemma \ref{max-u} identifies the maximal $F$-admissible symbols for $T$ with the  \textit{strongly disjoint set of bouquets} of $T$ introduced by Kimura in \cite[Definition 2.3]{K12}.  According to \cite[Theorem 4.1]{K12}, the maximum number of their \textit{flowers} gives the projective dimension of $R/I(T)$. In fact, the indices of the blocks of the maximal $F$-admissible symbols correspond to the roots of the bouquets and the other vertices (different from the roots) of the monomials remaining after the steps (I)-(V) correspond to the flowers of the bouquets. In particular, step (III) in the procedure above, together with the absence of gaps, ensures that the bouquets are pairwise strongly disjoint.
\end{rem}

From the previous remark and \cite[Theorem 4.1]{K12} it follows that:

\begin{cor}
Let $T$ be a tree. Then for every $i,j$, $\beta_{i,i+j}(R/I(T))$ is the number of subsets $W \subseteq V$ such that the induced subgraph $T_W$ contains an $F$-admissible symbol of length $i$ consisting of $j$ blocks.
\end{cor}
We can also recover Jacques'  recursive formulas for the graded Betti numbers of a tree \cite[Theorem 9.3.15]{J}, and for the projective dimension \cite[Theorem 9.4.17]{J}, for which we can present a new, short, elementary, non-homological proof. We first retrieve the necessary notation, which is the one introduced in \cite{J} on page 96, under (9.1.2). In \cite[Proposition 9.1.1]{J} it is proven that any tree $T$ contains  a vertex $v$ such that among its neighbours $v_1,\dots, v_n$ at most one (say $v_n$) has degree greater than 1.  Call $w_1, \dots, w_m$ the neighbours of $v_n$ other than $v$. Then call $T'$ the subgraph of $T$ obtained by eliminating the vertex $v_1$, and $T''$ the one obtained by eliminating all vertices $v, v_1, \dots, v_n$.  We can know prove Jacques' results.
\begin{thm} \cite[Chapter 9]{J} For all indices $r,d$ we have
$$\beta_{r,d}(R/I(T))= \beta_{r,d}(R/I(T'))+ \sum_{j=0}^{n-1}{{n-1}\choose{j}}\beta_{r-(j+1), d-(j+2)}(R/I(T'')).$$
Moreover,
$$\pd(R/I(T))=\max\{\pd(R/I(T')), \pd(R/I(T''))+n\}.$$
\end{thm}
\begin{proof}
We fix, on the vertices of $T$, an ordering with respect to which the smallest vertices are $v_n>v>v_1>\cdots>v_{n-1}$. Consider the set of $F$-admissible symbols for $T$ having length $r$ and degree $d$. It includes the set of all $F$-admissible symbols for $T'$ having length $r$ and degree $d$. The complementary set is formed by the $F$-admissible symbols containing the monomial $vv_1$. Let $u$ be such a symbol. Then $u'=u\setminus\{v_nv, vv_1,\dots, vv_{n-1}\}$ (if non-empty) is an $F$-admissible symbol for $T''$. Note that it is certainly reduced, and the elimination of $v_nv$ could produce a gap only for $n>1$ and if $u$ contained some monomial of the form $w_iv_n$. But $u$ cannot contain both $w_iv_n$ and $v_nv$ because by assumption it contains $vv_1$. Conversely, if $u'$ is an $F$-admissible symbol for $T''$ of length $r-(j+1)$ with $0\leq j\leq n-1$, then, for every $j$-subset $V$ of $\{vv_2,\dots, vv_n\}$, $u=u'\cup\{vv_1\}\cup V$ is an $F$-admissible symbol for $T$ of length $r$. Note that $u$ is reduced (because none of the vertices of the monomials added is a vertex of $T''$) and, moreover, if $n>1$ and some monomial of the form $w_ia$ appears in $u'$, then the presence of $vv_1$ prevents $w_ia$ and $v_nv$ from forming a gap. The symbol $u$ has degree $d$ if and only if $u'$ has degree $d-(j+2)$. Hence, given an $F$-admissible symbol $u'$ for $T''$ of length $r-(j+1)$ and degree $d-(j+2)$, we can construct the $F$-admissible symbol $u$ in ${n-1}\choose{j}$ different ways (as many as the possible choices of $V$). This shows the first identity. For the second identity it suffices to observe that the $F$-admissible symbol $u$ has maximum length if it is of maximum length among those not containing $vv_1$ or it is obtained by adding the whole set $\{v_nv, vv_1,\dots, vv_{n-1}\}$ to an $F$-admissible symbol $u'$ for $T''$ having maximum length.
\end{proof}

\section*{Acknowledgements}
The first author is gratefully indebted to Giandomenico Boffi for raising the question and for useful discussions. Both authors thank Volkmar Welker for his valuable hints on discrete Morse theory.

\end{document}